%% file: Almost_Free.tex
\documentclass[a4paper,12pt]{article}

\usepackage{amscd}
\usepackage{amsfonts}
\usepackage{amsmath}
\usepackage{amssymb}
\usepackage{amsthm}
\usepackage[T1]{fontenc} % changing encode
\usepackage{here} % [h] for tables
\usepackage{mathrsfs} % \mathscr
\usepackage{txfonts} % colneqq
\usepackage[all]{xy} % xypic
\usepackage{algorithm} % pseudocode
\usepackage{algpseudocode} %pseudocode
\usepackage{diagbox} % \diagbox

\allowdisplaybreaks

\theoremstyle{plain}
\newtheorem{thm}{Theorem}[section]
\newtheorem{lmm}[thm]{Lemma}
\newtheorem{prp}[thm]{Proposition}
\newtheorem{crl}[thm]{Corollary}

\theoremstyle{definition}
\newtheorem{dfn}[thm]{Definition}
\newtheorem{rmk}[thm]{Remark}
\newtheorem{exm}[thm]{Example}

\newcommand{\vs}[1][0.2]{\vspace{#1in}\noindent\ignorespaces}
\newcommand{\ba}{\begin{array*}}
\newcommand{\ea}{\end{array*}}
\newcommand{\be}{\begin{eqnarray*}}
\newcommand{\ee}{\end{eqnarray*}}
\newcommand{\bi}{\begin{itemize}}
\newcommand{\ei}{\end{itemize}}
\newcommand{\bb}{\vs\begin{itembox}}
\newcommand{\eb}{\end{itembox}}
\newcommand{\bc}{\begin{center}}
\newcommand{\ec}{\end{center}}
\newcommand{\bs}{\vs\begin{screen}}
\newcommand{\es}{\end{screen}}

\def\ens#1{{\mathchoice{\left\{ #1 \right\}}{\{ #1 \}}{\{ #1 \}}{\{ #1 \}}}}
\def\set#1#2{{\mathchoice{\left\{ #1 \ \middle| \ #2 \right\}}{\{ #1 \mid #2 \}}{\{ #1 \mid #2 \}}{\{ #1 \mid #2 \}}}}
\def\r#1{\text{\rm #1}}

\def\Bigv#1{\left| #1 \right|}
\def\v#1{{\mathchoice{\Bigv{#1}}{| #1 |}{| #1 |}{| #1 |}}}
\def\Bign#1{\left\| #1 \right\|}
\def\n#1{{\mathchoice{\Bign{#1}}{\| #1 \|}{\| #1 \|}{\| #1 \|}}}

\def\ol#1{\overline{#1}{}}

\newcommand{\bC}{\mathbb{C}}

\newcommand{\bN}{\mathbb{N}}

\newcommand{\bR}{\mathbb{R}}

\newcommand{\cF}{\mathscr{F}}

\newcommand{\cL}{\mathscr{L}}

\newcommand{\cP}{\mathscr{P}}

\newcommand{\cU}{\mathscr{U}}
\newcommand{\cV}{\mathscr{V}}

\newcommand{\rC}{\r{C}}

\newcommand{\rL}{\r{L}}

\newcommand{\rV}{\r{V}}

\newcommand{\C}{\bC}

\newcommand{\N}{\bN}

\newcommand{\R}{\bR}

\newcommand{\Cp}{\mathbb{C}_p}

\newcommand{\Fp}{\mathbb{F}_p}

\newcommand{\Qp}{\mathbb{Q}_p}

\newcommand{\Ban}{\r{Ban}}

\newcommand{\cf}{\r{cf}}
\newcommand{\ch}{\r{ch}}

\newcommand{\Card}{\r{Card}}

\newcommand{\Comp}{\r{Comp}}

\newcommand{\Hom}{\r{Hom}}

\newcommand{\Lim}{\r{Lim}}

\newcommand{\Ord}{\r{Ord}}

\newcommand{\rank}{\r{rank}}
\newcommand{\Reg}{\r{Reg}}

\newcommand{\Set}{\r{Set}}

\newcommand{\Suc}{\r{Suc}}

\newcommand{\ACF}{\r{ACF}}
\newcommand{\CF}{\r{CF}}
\newcommand{\FF}{\r{FF}}
\newcommand{\Norm}{\r{Norm}}

\algnewcommand\algorithmicbreak{{\bf break}}
\algnewcommand\Break{\algorithmicbreak{}}
\algnewcommand\algorithmiccontinue{{\bf continue}}
\algnewcommand\Continue{\algorithmiccontinue{}}

\title{Almost Free Non-Archimedean Banach Spaces and Relation to Large Cardinals}
\author{Tomoki Mihara}
\date{}

\begin{document}

\maketitle
%\address
\input{Abstract}
\tableofcontents
%\fn{03E55}{non-Archimedean analysis, almost free, compact cardinal}

\input{Introduction}
\input{Convention}
\input{Preliminaries}
\input{Free_Filtration}
\input{aleph_1_Strongly_Compact_Case}
\input{Weakly_Compact_Case}

\input{References}

\end{document}

%% file: Abstract.tex
\begin{abstract}
Let $k$ be a complete valuation field. We define freeness of a Banach $k$-vector space as the existence of an orthonormal Schauder basis, and almost freeness of a Banach $k$-vector space as a non-Archimedean Banach space analogue of almost freeness of an Abelian group. As non-Archimedean Banach space analogues of the classical facts that an almost free Abelian group is free under the assumption that its cardinality is $\aleph_1$-strongly compact or weakly compact, we show that an almost free Banach $k$-vector space is free under similar assumptions.
\end{abstract}

%% file: Introduction.tex
\section{Introduction}
\label{Introduction}

An Abelian group is said to be {\it $\kappa$-free} for a cardinal $\kappa$ (in the sense of \cite{Fuc15} \S 3.8 and \cite{CO25}) if every its subgroup of rank $< \kappa$ is free. We note that $\kappa$-freeness for a possibly singular cardinal $\kappa$ is defined also in another way (cf.\ the two paragraphs after \cite{EM02} Definition IV/1.1), but we omit the detail because we will focus on non-Archimedean Banach space analogues of results for the case where $\kappa$ is regular in this paper.

\vs
Every free Abelian group is $\kappa$-free for any cardinal $\kappa$ (cf.\ \cite{Fuc15} Theorem 3/1.6). For any cardinals $\kappa$ and $\lambda$ with $\lambda < \kappa$, every $\kappa$-free Abelian group is $\lambda$-free by definition. Therefore, the notion of $\kappa$-freeness defines a linear hierarchy of Abelian groups toward free Abelian groups.

\vs
Although there have been various studies on the hierarchy (cf.\ \cite{Fuc15} \S 3.8), we concentrate on one simple topic among them: ``Let $\kappa$ be an uncountable cardinal. When $\kappa$-freeness implies freeness?'' The question is much deeper than it looks.

\vs
P.\ Hills verified in \cite{Hil70} Theorem 3 that if $\cf(\kappa) = \omega$, then every $\kappa$-free Abelian group of cardinality $\kappa$ is free. Shelah's singular compactness theorem, which is named after the study of Whitehead's problem by S.\ Shelah in \cite{She74}, generalises the result for an axiomatically generalised notion of freeness (cf.\ \cite{Ekl06} \S 0.1 and \S 1.4). As a specialisation, it implies that if $\kappa$ is singular, then every $\kappa$-free Abelian group of cardinality $\kappa$ is free (cf.\ \cite{Fuc15} Theorem 3/9.2).

\vs
The question for the case where $\kappa$ is regular is not completed, and is closely related to studies of large cardinals. For example, J.\ Gregory verified in \cite{Gre73} Corollary of Theorem 1 that under the axiom $\rV = \rL$ of constructibility, there exists a non-free $\kappa$-free Abelian group of cardinality $\kappa$ if (and only if) $\kappa$ is not weakly compact.

\vs
The same question arises when we consider notions alternative to freeness, e.g.\ $\Sigma$-cyclicness. As a preceding study, F.\ Calderoni and A.\ Ostrem focused on the following two classical results on freeness, and verified analogues for the $\Sigma$-cyclicness (cf.\ \cite{CO25} Theorem 1.1 and Theorem 1.2):
\bi
\item[(1)] If $\kappa$ is $\aleph_1$-strongly compact, then every $\kappa$-free Abelian group of cardinality $\kappa$ is free. (cf.\ \cite{EM02} Theorem II/3.10 for an extension to a $\lambda$-strongly compact cardinal and a module over a PID of cardinality $< \lambda$ for an uncountable cardinal $\lambda$)
\item[(2)] If $\kappa$ is weakly compact, then every $\kappa$-free Abelian group of cardinality $\kappa$ is free. (cf.\ \cite{EM02} Theorem IV/3.2 for an extension to a $\leq \kappa$-generated module over a ring)
\ei
The paper \cite{CO25} compactly deals with the topics so that the reader not familiar with large cardinal axioms can learn about the relation between almost free Abelian group theory and compactness in large cardinal axioms. Therefore, it is valuable for an author of an analogous study to follow the structure of \cite{CO25}.

\vs
The aim of this paper is to imitate the structure of \cite{CO25} in the study of non-Archimedean analysis, i.e.\ analysis of Banach spaces over a complete valuation field $k$. Precisely speaking, we define a non-Archimedean Banach space analogue of freeness as the existence of an orthonormal Schauder basis, and verify corresponding results in a way completely parallel to the proofs of \cite{CO25} Theorem 1.1 and Theorem 1.2 based on methods in \cite{EM02}, except for conventional differences. Here are the statements of our main theorems:

\begin{thm}[non-Archimedean Banach space analogue of (1)]
\label{strongly compact intro}
Let $V$ be a Banach $k$-vector space with $\# V = \rank(V)$. If $\rank(V)$ is regular and $\lambda_k$-strongly compact, where $\lambda_k$ denotes the successor cardinal of $\max \ens{\#(k/O_k),\aleph_0}$, then the following are equivalent:
\bi
\item[(1)] The Banach $k$-vector space $V$ is free.
\item[(2)] The Banach $k$-vector space $V$ is almost free.
\ei
\end{thm}

\begin{thm}[non-Archimedean Banach space analogue of (2)]
\label{weakly compact intro}
Let $V$ be a Banach $k$-vector space. If $\# V = \rank(V)$ and $\rank(V)$ is weakly compact, then the following are equivalent:
\bi
\item[(1)] The Banach $k$-vector space $V$ is free.
\item[(2)] The Banach $k$-vector space $V$ is almost free.
\ei
\end{thm}

See Proposition \ref{trivial} and Proposition \ref{local field} for characterisation of freeness for the case where $k$ is a trivial valuation field or a local field, and Example \ref{Reid} for examples of non-free Banach $k$-vector spaces for the case where the valuation of $k$ has dense image under the assumption of the non-existence of a measurable cardinal.

\vs
We briefly explain contents of this paper. In \S \ref{Convention}, we introduce conventions for this paper. In \S \ref{Preliminaries}, we recall basic properties of Banach $k$-vector spaces. In \S \ref{Free Filtration}, we introduce the notion of freeness and almost freeness of a Banach $k$-vector space, and characterise almost freeness in terms of a filtration by free closed $k$-linear subspaces. In \S \ref{aleph_1-strongly Compact Case}, we verify Theorem \ref{strongly compact intro}. In \S \ref{Weakly Compact Case}, we verify Theorem \ref{weakly compact intro}.

%% file: Convention.tex
\section{Conventions}
\label{Convention}

We denote by $\Set$ the class of sets, by $\Ord \subset \Set$ the subclass of ordinals, by $\Suc \subset \Ord$ the subclass of successor ordinals, by $\Lim \subset \Ord$ the subclass of non-zero limit ordinals, by $\Card \subset \Ord$ the subclass of cardinals, and by $\Reg \subset \Card$ the subclass of uncountable regular cardinals.

\vs
We denote by $\aleph_0$ the least infinite cardinal, which is identical to the set $\N$ of non-negative integers and the least infinite ordinal $\omega$, and by $\aleph_1$ the least uncountable cardinal. For $\alpha \in \Suc$, we denote by $\alpha^{-}$ the predecessor $\max \alpha$ of $\alpha$.

\vs
For a class $X$ and $Y \in \Set$, we denote by $X^Y$ the class of maps $Y \to X$. When we handle a sequence or a family $s$ indexed by a set $I$, we frequently use the map notation $s(i)$ instead of the subscript notation $s_i$ to point the entry at $i \in I$, in order to avoid massive use of subscripts.

\vs
For a map $f$ and a subset $X'$ of its domain, we denote by $f \upharpoonright X'$ the restriction of $f$ to $X'$, and by $f[X']$ the image of $X'$ by $f$. For a map $f$ and $Y' \in \Set$, we denote by $f^{-1}[Y']$ the inverse image of $Y'$ by $f$.

\vs
For $X \in \Set$, $x \in X$, and a binary relation $R$ on $X$, we set $X_{R x} \coloneqq \set{x' \in X}{x' R x}$. We note that every $d \in \N$ is identical to $\omega_{< d}$, and hence for a set $X$, $X^d$ formally means $X^{\omega_{< d}}$, which is naturally identified with the set of $d$-tuples in $X$.

\vs
For $X \in \Set$ and a map $f \colon X \to \R_{\geq 0}$, we denote by $\sup_{x \in X} f(x)$ the supremum of the image of $f$ in $\R_{\geq 0} \sqcup \ens{\infty}$. In particular, $\sup_{x \in X} f(x)$ for the case $X = \emptyset$ is $0$ rather than $- \infty$ in this context.

\vs
For $X \in \Set$, we denote by $\# X$ its cardinality, and by $\cP(X)$ the set of subsets of $X$ partially ordered by inclusions. For $X \in \Set$, $\kappa \in \Card$, and a binary relation $R$ on $\Card$, we set $\cP_{R \kappa}(X) \coloneqq \set{X' \in \cP(X)}{\# X' R \kappa}$. We say that $\kappa \in \Card$ is {\it weakly compact} if $\kappa \geq \aleph_1$ holds and for any $f \in \ens{0,1}^{\cP_{= 2}(\kappa)}$, there exists some $S \in \cP_{= \kappa}(\kappa)$ such that $\#(f[\cP_{= 2}(S)]) = 1$.

\vs
When we refer to a partially ordered class, we always regard it as a category in the canonical way. A partially ordered set $I$ is said to be {\it $\kappa$-directed} for $\kappa \in \Card$ if for any $I' \in \cP_{< \kappa}(I)$, there exists some $i \in I$ such that for any $i' \in I'$, $i' \leq i$ holds. In particular, $\aleph_0$-directedness is equivalent to directedness in the usual sense.

\vs
A filter $\cF$ is said to be {\it $\lambda$-complete} for $\lambda \in \Card$ if $\bigcap_{U \in S} U \in \cF$ holds for any $S \in \cP_{< \lambda}(\cF) \setminus \ens{\emptyset}$. We say that $\kappa \in \Card$ is {\it measurable} if $\kappa \geq \aleph_1$ holds and $\kappa$ admits some $\kappa$-complete non-principal ultrafilter, and is {\it $\lambda$-strongly compact} for $\lambda \in \Card$ if $\kappa \geq \aleph_1$ holds and for any $I \in \Set$ and any $\kappa$-complete filter $\cF$ on $I$, there exists some $\lambda$-complete ultrafilter on $I$ containing $\cF$.

\vs
We note that $\lambda$-strong compactness for $\lambda \in \Card$ does not imply the regularity, unlike usual large cardinals (cf.\ the paragraph after \cite{EM02} Definition II/2.17). The reader should be careful that $\lambda$-strong compactness for $\lambda \in \Card$ in \cite{CO25} p.\ 3 is defined implicitly under the assumption of regularity, because the definition of $\kappa$-completeness for $\kappa \in \Card$ is defined under the assumption of the regularity of $\kappa$.

\vs
A {\it complete valuation field} is a field $k$ equipped with a map $\v{\cdot} \colon k \to \R_{\geq 0}$ called a {\it (multiplicative) valuation} satisfying the following:
\bi
\item[(1)] For any $c \in k$, $\v{c} = 0$ holds if and only if $c = 0$ holds.
\item[(2)] For any $(c_0,c_1) \in k^2$, $\v{c_0 - c_1} \leq \max \ens{\v{c_0},\v{c_1}}$ holds.
\item[(3)] For any $(c_0,c_1) \in k^2$, $\v{c_0 c_1} = \v{c_0} \ \v{c_1}$ holds.
\item[(4)] The ultrametric on $k$ defined by
\be
(c_0,c_1) \mapsto \v{c_0 - c_1}
\ee
is complete.
\ei
The reader should be careful not to confuse the notations of the valuation $\v{\cdot}$ and the cardinality $\#$.

\vs
Throughout this paper, $k$ denotes a complete valuation field. We denote by $O_k$ the valuation ring $\set{c \in k}{\v{c} \leq 1}$ of $k$. We say that the valuation of $k$ is {\it trivial} if $\v{k^{\times}} = \ens{1}$, and is {\it discrete} if $\v{k^{\times}}$ is a free subgroup of rank $1$ of the multiplication group $\R_{> 0}$. We say that $k$ is a {\it local field} if the valuation of $k$ is discrete and the residue field of the local ring $O_k$ is a finite field.

\vs
A {\it normed $k$-vector space} is a $k$-vector space $V$ equipped with a map $\n{\cdot} \colon V \to \R_{\geq 0}$ called a {\it norm} satisfying the following:
\bi
\item[(1)] For any $v \in V$, $\n{v} = 0$ holds if and only if $v = 0$ holds.
\item[(2)] For any $(v_0,v_1) \in V^2$, $\n{v_0 - v_1} \leq \max \ens{\n{v_0},\n{v_1}}$ holds.
\item[(3)] For any $(c,v) \in k \times V$, $\n{cv} = \v{c} \ \n{v}$ holds.
\ei
A {\it Banach $k$-vector space} is a normed $k$-vector space satisfying the following:
\bi
\item[(4)] The ultrametric on $V$ defined by
\be
(v_0,v_1) \mapsto \n{v_0 - v_1}
\ee
is complete.
\ei
We say that the norm of a normed $k$-vector space $V$ is {\it trivial} if $\n{V} \subset \ens{0,1}$. We note that every $0$-dimensional $k$-vector space forms a normed $k$-vector space with respect to the unique trivial norm, and if the valuation of $k$ is trivial, then every $k$-vector space forms a normed $k$-vector space with respect to the unique trivial norm. On the other hand, the norm of a normed $k$-vector space is not necessarily trivial even when the valuation of $k$ is trivial. For example, for any prime number $p$, the discrete valuation of the formal power series algebra $\Fp[[T]]$ is a non-trivial norm on an $\Fp$-vector space, where $\Fp$ is equipped with the trivial valuation. We denote by $\Norm(k)$ the class of normed $k$-vector spaces, and by $\Ban(k) \subset \Norm(k)$ the subclass of Banach $k$-vector spaces.

\vs
Let $(V,W) \in \Norm(k)^2$. A $k$-linear homomorphism $f \colon V \to W$ is said to be {\it bounded} if there exists some $C \in \R_{\geq 0}$ such that for any $v \in V$, $\n{f(v)} \leq C \n{v}$ holds, and is said to be {\it contracting} if such a $C$ can be chosen to be $1$. We denote by $\n{f}$ the infimum of such a $C$, and call it {\it the operator norm of $f$}. We denote by $\Hom(V,W)$ the Banach $k$-vector space of bounded $k$-linear homomorphisms $V \to W$ equipped with the operator norm, and by $\Hom_{\leq 1}(V,W)$ its subset of contracting $k$-linear homomorphisms.

\vs
Whenever we refer to $\Norm(k)$ or its subclass in this paper, we equip it with the category structure given by $\Hom_{\leq 1}$. In particular, a morphism of normed $k$-vector spaces always refers to a contracting $k$-linear homomorphism, and an isomorphism of normed $k$-vector spaces always refers to an isomorphism in $\Norm(k)$ in this sense, i.e.\ an isometric $k$-linear isomorphism.

\vs
Let $V_0$ be a closed $k$-linear subspace of $V \in \Ban(k)$. Then $V_0$ forms a Banach $k$-vector space with respect to the restriction of the structure of $V$, and the quotient $k$-vector space $V/V_0$ also forms a Banach $k$-vector space with respect to the quotient norm
\be
\ol{v} & \mapsto & \inf_{v \in \ol{v}} \n{v}.
\ee
We always regard $V_0$ and $V/V_0$ as Banach $k$-vector spaces in these ways.

\vs
Let $I \in \Set$ and $\cV \in \Ban(k)^I$. We denote by $\prod_{i \in I} \cV(i)$ the bounded direct product of $\cV$, i.e.\ the Banach $k$-vector space whose underlying set is the set of maps $v \colon I \to \bigsqcup_{i \in I} \cV(i)$ such that $v(i) \in \cV(i)$ for any $i \in I'$ and $\sup_{i \in I} \n{v(i)} < \infty$ and whose norm is the supremum norm, i.e.\ the map $\n{\cdot} \colon \prod_{i \in I} \cV(i) \to \R_{\geq 0}$ defined by
\be
\n{v} \coloneqq \sup_{i \in I} \n{v(i)},
\ee
and by $\bigoplus_{i \in I} \cV(i)$ the completed direct sum of $\cV$, i.e.\ the closed $k$-linear subspace of $\prod_{i \in I} \cV(i)$ given as
\be
& & \set{v \in \prod_{i \in I} \cV(i)}{\forall \epsilon \in \R_{> 0}[\exists I_0 \in \cP_{< \aleph_0}(I)[\forall i \in I \setminus I_0[\n{v(i)} < \epsilon]]]} \\
& = & \set{v \in \prod_{i \in I} \cV(i)}{\forall \epsilon \in \R_{> 0} \left[ \# \set{i \in I}{\n{v(i)} \geq \epsilon} < \aleph_0 \right]}.
\ee
We do not use $\prod$ and $\bigoplus$ for algebraic direct products or algebraic direct sums.

\vs
For $I \in \Set$ and $V \in \Ban(k)$, we denote by $\ell^{\infty}(I,V)$ (resp.\ $\rC_0(I,V)$) the bounded direct product $\prod_{i \in I} V$ (resp.\ the completed direct sum $\bigoplus_{i \in I} V$) of the constant family $(V)_{i \in I}$.

\vs
For $I \in \Set$, $V \in \Ban(k)$, and $v \in \rC_0(I,V)$, we denote by $\sum_{i \in I} v(i) \in V$ the usual sum of $v$ if $\# I < \aleph_0$, and otherwise the unique $s \in V$ satisfying that for any $\epsilon \in \R_{> 0}$, there exists some $I_0 \in \cP_{< \aleph_0}(I)$ such that $\n{v(i)} < \epsilon$ holds for any $i \in I \setminus I_0$ and $\n{s - \sum_{i \in I_0} v(i)} < \epsilon$ holds.

%% file: Preliminaries.tex
\section{Preliminaries}
\label{Preliminaries}

We recall basic properties of $\Ban(k)$. We denote by $\Comp$ the completion functor $\Norm(k) \to \Ban(k)$ and by $F_{\Comp}$ the forgetful functor $\Ban(k) \hookrightarrow \Norm(k)$. We first recall colimits in $\Norm(k)$ and $\Ban(k)$.

\vs
Let $I$ be a directed set, and $\cV$ a functor $I \to \Norm(k)$. The colimit $\varinjlim \cV \in \Norm(k)$ of $\cV$ is given as the normed $k$-vector space whose underlying $k$-vector space is the quotient $V'/W$ of the colimit $V'$ of the underlying $k$-vector spaces by its $k$-vector subspace $W$ consisting of $v \in V'$ such that for any $\epsilon \in \R_{> 0}$, there exists a pair $(i,v')$ of $i \in I$ and $v' \in \cV(i)$ such that $\n{v'} < \epsilon$ and the image of $v'$ in $V'$ is $v$, and whose norm assigns to each $\ol{v} \in V'/W$ the infimum of $\n{v'}$ with $i \in I$ and $v' \in \cV(i)$ such that the image of $v'$ in $V'/W$ is $\ol{v}$. Given $i \in I$ and $v' \in \cV(i)$, if we denote by $v_{i'}$ the image of $v'$ in $\cV(i')$ for each $i' \in I_{\geq i}$ and by $v$ the image of $v'$ in $V'$, then we have
\be
\n{v + W} = \inf_{i' \in I_{\geq i}} \n{v_{i'}}.
\ee
Since $\Comp$ is left adjoint to $F_{\Comp}$, the colimit $\varinjlim (\Comp \circ \cV) \in \Ban(k)$ of $\Comp \circ \cV$ is given as the completion of $\varinjlim \cV$. In particular, if $\cV$ takes values in $\Ban(k)$, its colimit in $\Ban(k)$ is given as the completion of its colimit in $\Norm(k)$. We give a specialised description of its colimit in $\Ban(k)$ when $I$ is $\aleph_1$-directed.

\begin{prp}
\label{completeness of aleph_1-directed limit}
Let $I$ be an $\aleph_1$-directed set, and $\cV$ a functor $I \to \Ban(k)$. Set $V \coloneqq \varinjlim (F_{\Comp} \circ \cV)$. Then the following hold:
\bi
\item[(1)] For any $i \in I$ and $v \in \cV(i)$, the norm of the image of $v$ in $V$ is the minimum of the norms of the images of $v$ in $\cV(i')$ with $i' \in I_{\geq i}$.
\item[(2)] The equality $F_{\Comp}(\varinjlim \cV) = V$ holds.
\ei
\end{prp}

\begin{proof}
For each $i \in I$, we denote by $\phi_i$ the canonical morphism $\cV(i) \to V$. For each $(i,i') \in I^2$ with $i \leq i'$, we denote by $\phi_{i,i'}$ the given morphism $\cV(i) \to \cV(i')$. 

\vs
(1) For any $i' \in I_{\geq i}$, we have $\n{\phi_i(v)} = \n{\phi_{i'}(\phi_{i,i'}(v))} \leq \n{\phi_{i,i'}(v)}$. For any $\epsilon \in \R_{> 0}$, there exists some $i' \in I_{\geq i}$ such that $\n{\phi_{i,i'}(v)} < \n{\phi_i(v)} + \epsilon$ by the definition of the colimit in $\Norm(k)$. By the first countability of $\R_{\geq 0}$ and the $\aleph_1$-directedness of $I$, there exists some $i' \in I_{\geq i}$ such that $\n{\phi_{i,i'}(v)} \leq \n{\phi_i(v)}$. This implies the assertion.

\vs
(2) It suffices to show the completeness of $V$. Let $v \in V^{\omega}$ be a Cauchy sequence. We show that $v$ converges in $V$. By (1) and the $\aleph_1$-directedness of $I$, there exists a pair $(i,v')$ of $i \in I$ and $v' \in \cV(i)^{\omega}$ such that $\phi_i \circ v' = v$ and $\n{v'(h_0) - v'(h_1)} = \n{v(h_0) - v(h_1)}$ for any $(h_0,h_1) \in \omega^2$. In particular, $v'$ is a Cauchy sequence in $\cV(i)$, and hence converges to some $v'_{\infty} \in \cV(i)$ by the completeness of $\cV(i)$. By the boundedness of $\phi_i$, $v = \phi_i \circ v'$ converges to $\phi_i(v'_{\infty}) \in V$.
\end{proof}

For $V \in \Ban(k)$ and $S \in \cP(V)$, we denote by $\langle S \rangle \subset V$ the closure of the $k$-linear subspace generated by $S$.

\begin{dfn}
For $V \in \Ban(k)$, we denote by $\rank(V)$ the minimum of the cardinality of $S \in \cP(V)$ such that $V = \langle S \rangle$, and say that $V$ is {\it of countable type} if $\rank(V) \leq \aleph_0$.
\end{dfn}

For $V \in \Ban(k)$, $E \in \cP(V)$ is said to be an {\it orthonormal Schauder basis of $V$} if it satisfies the following:
\bi
\item[(1)] For any $e \in E$, $\n{e} = 1$ holds.
\item[(2)] For any $c \in \rC_0(E,k)$, $\n{\sum_{e \in E} c(e) e} = \n{c}$ holds, where the left hand side makes sense because (1) implies $(c(e) e)_{e \in E} \in \rC_0(E,V)$.
\item[(3)] For any $v \in V$, there exists some $c \in \rC_0(E,k)$ such that $\sum_{e \in E} c(e) e = v$, where $c$ is unique by (2).
\ei
We note that $E \in \cP(V)$ is an orthonormal Schauder basis if and only if the canonical $k$-linear homomorphism $\rC_0(E,k) \to V$ associated to the inclusion $E \hookrightarrow V$ is an isomorphism. For example, for any $I \in \Set$, the canonical basis of $\rC_0(I,k)$, i.e.\ the subset of characteristic functions of singletons, is an orthonormal Schauder basis. Unlike a basis of a $k$-vector space, an orthonormal Schauder basis of a Banach $k$-vector space does not necessarily exist. Nevertheless, all orthonormal Schauder bases for a given Banach $k$-vector space have the same
cardinality, as shown by the following:

\begin{prp}
\label{uniqueness of dimension}
Let $V \in \Ban(k)$. Every orthonormal Schauder basis $E$ of $V$ is of cardinality $\rank(V)$.
\end{prp}

\begin{proof}
Set $\kappa \coloneqq \rank(V)$. By \cite{BGR84} Proposition 2.3.3/4, every finite dimensional $k$-linear subspace of $V$ is closed. If $\dim_k V < \aleph_0$, then $E$ is a $k$-linear basis of $V$ and $\kappa = \dim_k V$ holds because every $k$-linear subspace of $V$ is closed. Suppose $\dim_k V \geq \aleph_0$. Then we have $\# E \geq \aleph_0$ because otherwise $E$ generates a finite dimensional closed $k$-linear subspace. By the definition of $\kappa$, we have $\kappa \leq \# E$. Therefore, it suffices to show $\# E \leq \kappa$.

\vs
Since $V$ is isomorphic to $\rC_0(E,k)$, we have $\kappa = \rank(\rC_0(E,k))$. Take $S \in \cP_{= \kappa}(\rC_0(E,k))$ such that $\rC_0(E,k) = \langle S \rangle$. Let $c \in S$. Set $E_c \coloneqq \set{e \in E}{c(e) \neq 0}$. By the definition of $\rC_0(E,k)$, we have $\# E_c \leq \aleph_0$.

\vs
Set $E' \coloneqq \bigcup_{c \in S} E_c$. We show $E' = E$. We have $E' \subset E$ by the definition of $E'$. If the equality does not hold, then the closure $\set{c \in \rC_0(E,k)}{\forall e \in E \setminus E'[c(e) = 0]}$ of the $k$-linear subspace of $\rC_0(E,k)$ generated by the image of $E'$ does not contain the image of $E \setminus E' \neq \emptyset$. Since $S$ generates a dense $k$-linear subspace of $\rC_0(E,k)$, so does the image of $E'$ by the definition of $E'$. Therefore, the equality holds. We obtain
\be
\kappa \leq \# E = \# E' \leq \sum_{c \in S} \# E_c \leq \sum_{c \in S} \aleph_0 = \# S \times \aleph_0 = \kappa \times \aleph_0 = \kappa
\ee
with respect to cardinal arithmetic by $\# S = \kappa \geq \aleph_0$.
\end{proof}

For $V \in \Ban(k)$ and $\kappa \in \Card$, we set
\be
\cL_{< \kappa}(V) \coloneqq \set{\langle S \rangle}{S \in \cP_{< \kappa}(V)},
\ee
and regard it as a partially ordered set by inclusions.

\begin{prp}
\label{L is directed}
For any $V \in \Ban(k)$ and $\kappa \in \Card \setminus \omega$, $\cL_{< \kappa}(V)$ is $\cf(\kappa)$-directed.
\end{prp}

\begin{proof}
Let $\cV \subset \cP_{< \cf(\kappa)}(\cL_{< \kappa}(V))$. For each $V' \in \cV$, take $S_{V'} \in \cP_{< \kappa}(V)$ with $V' = \langle S_{V'} \rangle$. Set $S \coloneqq \bigcup_{V' \in \cV} S_{V'}$ and $V_0 \coloneqq \langle S \rangle$. By $\# \cV < \cf(\kappa)$ and $\# S_{V'} < \kappa$ for any $V' \in \cV$, we have $\# S < \kappa$. Therefore, we have $V_0 \in \cL_{< \kappa}(V)$. For any $V' \in \cV$, we have $S_{V'} \subset S$ and hence $V' \subset V_0$.
\end{proof}

Although we do not use the following, it is helpful for the reader to understand the structure of $\Ban(k)$:

\begin{crl}
\label{accessibility}
Let $\lambda$ denote the least cardinal larger than both of $\aleph_0$ and the minimum of the cardinalities of dense subsets of $k$. Then for any $\kappa \in \Reg \setminus \lambda$, every Banach $k$-vector space of rank $< \kappa$ is a $\kappa$-compact object of $\Ban(k)$, and $\Ban(k)$ is a $\kappa$-accessible category.
\end{crl}

\begin{proof}
First, we show that every $V \in \Ban(k)$ with $\rank(V) < \kappa$ is a $\kappa$-compact object of $\Ban(k)$ with $\# V \leq \kappa^{\aleph_0}$ with respect to the cardinal arithmetic. By $\lambda \leq \kappa$, $V$ admits a dense subset $S$ of cardinality $< \kappa$. Let $I$ be a $\kappa$-directed set, $\cV \colon I \to \Ban(k)$ a functor, and $\phi$ a morphism $V \to \varinjlim \cV$. By Proposition \ref{completeness of aleph_1-directed limit} (2), we have $F_{\Comp}(\varinjlim \cV) = \varinjlim (F_{\Comp} \circ \cV)$. By Proposition \ref{completeness of aleph_1-directed limit} (1) and $\# S < \kappa$, there exists some pair $(i,\phi')$ of $i \in I$ and $\phi' \in \cV(i)^S$ satisfying the following:
\bi
\item[(1)] For any $v \in S$, the image of $\phi'(v)$ in $\varinjlim \cV$ is $\phi(v)$.
\item[(2)] For any $(v_0,v_1) \in S^2$, $\n{\phi'(v_0) - \phi'(v_1)} = \n{\phi(v_0) - \phi(v_1)}$ holds.
\ei
Since $\phi$ sends any Cauchy sequence in $V$ to a Cauchy sequence in $\varinjlim V$ by the boundedness of $\phi$, the properties (1) and (2) imply that $\phi'$ sends any Cauchy sequence in $S$ to a Cauchy sequence in $\cV(i)$. Since every element of $V$ is a limit of Cauchy sequence in $S$, we obtain $\# V \leq \kappa^{\aleph_0}$ and $\phi$ factors through $\cV'(i)$. Therefore, $V$ is a $\kappa$-compact object of $\Ban(k)$.

\vs
The set of $V \in \Ban(k)$ with $\rank(V) < \kappa$ whose underlying sets are subsets of $\kappa^{\aleph_0}$ generates $\Ban(k)$ under $\kappa$-directed small colimits by Proposition \ref{completeness of aleph_1-directed limit} (2) and Proposition \ref{L is directed}. Therefore, $\Ban(k)$ is a $\kappa$-accessible category.
\end{proof}

\begin{crl}
If $k$ is separable, then every Banach $k$-vector space of countable type is an $\aleph_1$-compact object of $\Ban(k)$, and  $\Ban(k)$ is an $\aleph_1$-accessible category.
\end{crl}

\begin{proof}
The assertion follows from Corollary \ref{accessibility} by $\lambda = \aleph_1$.
\end{proof}

The following is an ultrapower construction of a non-Archimedean Banach space, which is well-known for Archimedean Banach spaces, i.e.\ Banach spaces over $\R$ or $\C$.

\begin{prp}
\label{ultrapower}
Let $I \in \Set$ with an ultrafilter $\cU$ on $I$, and $\cV \in \Ban(k)^I$. Set $V \coloneqq \prod_{i \in I} \cV(i)$ and $V_0 \coloneqq \set{v \in V}{\forall \epsilon \in \R_{> 0}[\set{i \in I}{\n{v(i)} < \epsilon} \in \cU]}$. Then the following hold:
\bi
\item[(1)] The set $V_0$ is a closed $k$-linear subspace of $V$.
\item[(2)] For any $v \in V$, the norm $\n{v + V_0}$ of $v + V_0 \in V/V_0$ is a unique $r \in \R_{\geq 0}$ such that for any $\epsilon \in \R_{> 0}$, $\set{i \in I}{\v{\n{v(i)} - r} < \epsilon} \in \cU$ holds.
\ei
\end{prp}

\begin{proof}
The proof is completely parallel to the Archimedean counterpart. Since the proof tends to be omitted, we write down the proof for the reader's convenience. For any $(c,v) \in k \times V_0$, we have
\be
\set{i \in I}{\n{(cv)(i)} < \epsilon} = \set{i \in I}{\v{c} \ \n{v(i)} < \epsilon} = \set{i \in I}{\n{v(i)} < \v{c}^{-1} \epsilon} \in \cU
\ee
for any $\epsilon \in \R_{> 0}$, where $\v{0}^{-1} \epsilon$ is formally defined as $\infty$, and hence $cv \in V_0$. For any $(v_0,v_1) \in V_0^2$, we have
\be
\set{i \in I}{\n{(v_0+v_1)(i)} < \epsilon} \supset \bigcap_{v \in \ens{v_0,v_1}} \set{i \in I}{\n{v(i)} < \epsilon} \in \cU
\ee
for any $\epsilon \in \R_{> 0}$, and hence $v_0 + v_1 \in V_0$. For any $v \in V$ in the closure of $V_0$, for any $\epsilon \in \R_{> 0}$, there exists some $v' \in V_0$ such that $\n{v - v'} < \epsilon$, which implies
\be
\set{i \in I}{\n{v(i)} < \epsilon} = \set{i \in I}{\n{v'(i)} < \epsilon} \in \cU,
\ee
and hence $v \in V_0$. Therefore, $V_0$ is a closed $k$-linear subspace of $V$.

\vs
Let $v \in V$. Set $r \coloneqq \n{v + V_0}$. Let $\epsilon \in \R_{> 0}$. By the definition of $r$, there exists some $v' \in V_0$ such that $r \leq \n{v + v'} < r + 2^{-1} \epsilon$. We have
\be
\set{i \in I}{\v{\n{v(i)} - r} < \epsilon} \supset \set{i \in I}{\n{v'(i)} < 2^{-1} \epsilon} \in \cU.
\ee
This implies that $r$ satisfies the condition in the assertion (2). For any $r' \in \R_{\geq 0} \setminus \ens{r}$, we have
\be
\set{i \in I}{\v{\n{v(i)} - r'} < 2^{-1}(r - r')} \subset I \setminus \set{i \in I}{\v{\n{v(i)} - r} < 2^{-1}(r - r')}  \notin \cU,
\ee
and hence $r'$ does not satisfy the same condition.
\end{proof}

In the setting in Proposition \ref{ultrapower}, we set $(\prod_{i \in I} \cV(i))/\cU \coloneqq V/V_0$ and denote by $[v]_{\cU}$ the equivalence class of $v$ for each $v \in V$. We give a simpler description of the ultraproduct for an $\aleph_1$-complete filter.

\begin{crl}
\label{equivalence relation}
Let $I \in \Set$ with an $\aleph_1$-complete ultrafilter $\cU$ on $I$, and $\cV \in \Ban(k)^I$. Set $V \coloneqq \prod_{i \in I} \cV(i)$ and $V_0 \coloneqq \set{v \in V}{\set{i \in I}{v(i) = 0} \in \cU}$. Then the following hold:
\bi
\item[(1)] The set $V_0$ coincides with the closed $k$-linear subspace of $V$ in Proposition \ref{ultrapower}.
\item[(2)] For any $v \in V$, the norm $\n{v + V_0}$ of $v + V_0 \in V/V_0$ is a unique $r \in \R_{\geq 0}$ such that $\set{i \in I}{\n{v(i)} = r} \in \cU$ holds.
\ei
\end{crl}

\begin{proof}
By the first countability of $\R$ and the $\aleph_1$-completeness of $\cU$, we have
\be
\set{i \in I}{\n{v(i)} = r} \in \cU \Leftrightarrow \forall \epsilon \in \R_{> 0}[\set{i \in I}{\v{\n{v(i)} - r} < \epsilon} \in \cU],
\ee
which we denote by $E(r)$, for any $r \in \R_{\geq 0}$. By $E(0)$, we have
\be
V_0 = \set{v \in V}{\set{i \in I}{\n{v(i)} = 0} \in \cU} = \bigcap_{\epsilon \in \R_{> 0}} \set{v \in V}{\set{i \in I}{\n{v(i)} < \epsilon} \in \cU},
\ee
and hence the assertion (1) holds. The assertion (2) follows from Proposition \ref{ultrapower} (2) and $E(r)$ for all $r \in \R_{\geq 0}$.
\end{proof}

%% file: Free_Filtration.tex
\section{Free Filtration}
\label{Free Filtration}

In this section, we introduce the notion of a free filtration as a non-Archimedean Banach space analogue of a filtration of free subgroups or $\Sigma$-cyclic subgroups (cf.\ \cite{Fuc15} Lemma 3/8.6 and \cite{CO25} Proposition 4.6), and study its properties. For this purpose, we first introduce the notion of a free Banach $k$-vector space.

\begin{dfn}
We say that $V \in \Ban(k)$ is {\it free} if $V$ admits an orthonormal Schauder basis, is {\it $\kappa$-free} for $\kappa \in \Card$ if every $V' \in \cL_{< \kappa}(V)$ is free, and is {\it almost free} if $V$ is $\rank(V)$-free.
\end{dfn}

Freeness of a Banach $k$-vector space of countable type is characterised in terms of a strictly $k$-Cartesian normed $k$-vector space (cf.\ \cite{BGR84} Proposition 2.7.5/2). By definition, a Banach $k$-vector space is free if and only if it is isomorphic to $\rC_0(I,k)$ for some $I \in \Set$.

\vs
We note that an isomorphism in this paper always refers to an isometric $k$-linear isomorphism as declared in \S \ref{Convention}. If we formulated freeness in terms of an admissible isomorphism, i.e.\ a bounded $k$-linear isomorphism whose inverse is also bounded, then under the assumption of the non-triviality of the valuation of $k$, every Banach $k$-vector space of countable type would be free (cf.\ \cite{BGR84} Theorem 2.8.2/2 and \cite{Sch99} Theorem 6) and hence every Banach $k$-vector space would be $\aleph_1$-free. Therefore, the situation would be less interesting.

\vs
We give characterisations of freeness for two simple cases: the trivial valuation case and the local field case.

\begin{prp}
\label{trivial}
If the valuation of $k$ is trivial, then for any $V \in \Ban(k)$, the following are equivalent:
\bi
\item[(1)] The Banach $k$-vector space $V$ is free.
\item[(2)] The Banach $k$-vector space $V$ is almost free, and is isomorphic to $k$ if $\dim_k V = 1$.
\item[(3)] The inclusion $\n{V} \subset \v{k}$ holds.
\ei
\end{prp}

\begin{prp}
\label{local field}
If $k$ is a local field, for any $V \in \Ban(k)$, the following are equivalent:
\bi
\item[(1)] The Banach $k$-vector space $V$ is free.
\item[(2)] The Banach $k$-vector space $V$ is almost free, and is isomorphic to $k$ if $\dim_k V = 1$.
\item[(3)] The inclusion $\n{V} \subset \v{k}$ holds.
\ei
\end{prp}

\begin{proof}[Proof of Proposition \ref{trivial}]
We note that under the condition (3), i.e.\ $\n{V} \subset \v{k} = \ens{0,1}$, the norm of $V$ is trivial, and hence the notion of an orthonormal Schauder basis of $V$ is equivalent to that of a basis of the underlying $k$-vector space of $V$.

\vs
The equivalence between (1) and (3) follows from the definition of an orthonormal Schauder basis. We show that (1) is equivalent to the condition that every closed $k$-linear subspace $W \subset V$ with $\dim_k W \leq 1$ is free.

\vs
First, assume (1). Let $W$ be a closed $k$-linear subspace of $V$ with $\dim_k W \leq 1$. We have $\n{W} \subset \n{V} \subset \v{k}$, and hence $W$ is free by the equivalence between (1) and (3) applied to $W$.

\vs
Next, assume that every closed $k$-linear subspace $W \subset V$ with $\dim_k W \leq 1$ is free. Let $v \in V$. By \cite{BGR84} Proposition 2.3.3/4, every finite dimensional $k$-linear subspace of $V$ is closed, and hence $\langle \ens{v} \rangle$ is generated by $\ens{v}$. This implies that $\langle \ens{v} \rangle$ is a closed $k$-linear subspace of $V$ with $\dim_k \langle \ens{v} \rangle \leq 1$. By the assumption, $\langle \ens{v} \rangle$ is free. By the equivalence between (1) and (3) applied to $\langle \ens{v} \rangle$, we obtain $\n{\langle \ens{v} \rangle} \subset \v{k}$, and hence $\n{v} \in \v{k}$. This implies $\n{V} \subset \v{k}$. By the equivalence between (1) and (3) applied to $V$, $V$ is free.

\vs
Therefore, (1) is equivalent to (2) if $\dim_k V > 1$. If $\dim_k V = 1$, then (3) is equivalent to the condition that $V$ is isomorphic to $\rC_0(\ens{\ast},k) \cong k$. If $\dim_k V = 0$, then $V$ is free, and hence (1) and (2) hold.
\end{proof}

\begin{proof}[Proof of Proposition \ref{local field}]
The equivalence between (1) and (3) immediately follows from Schikhof duality (cf.\ \cite{ST02} Theorem 1.2 for the case $\ch(k) = 0$ and \cite{Mih21} Proposition 1.7 for a general case). The equivalence between (1) and (2) follows from the same argument as Proposition \ref{trivial}.
\end{proof}

By Proposition \ref{trivial} and Proposition \ref{local field}, we are mainly interested in the case where $\v{k^{\times}}$ is a dense subgroup of the multiplicative group $\R_{> 0}$.

\begin{exm}
\label{Reid}
Suppose $\v{k^{\times}} = \R_{> 0}$. If $V$ is one of
\be
\ell^{\infty}(\lambda_0,k), \rC_0(\lambda_1,\ell^{\infty}(\lambda_0,k)), \ell^{\infty}(\lambda_2,\rC_0(\lambda_1,\ell^{\infty}(\lambda_0,k))),
\ee
and so on for infinite cardinals $\lambda_0, \lambda_1, \lambda_2, \ldots$, then $\n{V} = \R_{> 0} = \v{k}$ holds. However, if there exists no measurable cardinal, then none of such a $V$ is free by \cite{Mih26} Theorem 4.10.
\end{exm}

We give a non-Archimedean Banach space analogue of the almost freeness of a free Abelian group, which follows from the freeness of a subgroup of a free Abelian group (cf.\ \cite{Fuc15} Theorem 3/1.6), and the almost $\Sigma$-cyclicness of a $\Sigma$-cyclic group, which follows from Kulikov's theorem, i.e.\ the $\Sigma$-cyclicness of a subgroup of a $\Sigma$-cyclic group (cf.\ \cite{Fuc15} Theorem 3/5.7).

\begin{prp}
\label{free implies almost free}
For any $V \in \Ban(k)$ and $\kappa \in \Card$, if $V$ is free, then $V$ is $\kappa$-free.
\end{prp}

\begin{proof}
The assertion follows from the freeness of a closed $k$-linear subspace of a free Banach $k$-vector space (cf.\ \cite{Sch99} Theorem 5 (ii)).
\end{proof}

Now we are ready for introducing a free filtration.

\begin{dfn}
Let $V \in \Ban(k)$. Set $\kappa \coloneqq \rank(V)$. A {\it free filtration of $V$} is $\cV \in \cL_{< \kappa}(V)^{\kappa}$ satisfying the following:
\bi
\item[(1)] For any $\alpha \in \kappa$, $\cV(\alpha)$ is free.
\item[(2)] For any $\alpha \in \kappa$ and $\beta \in \alpha$, $\cV(\beta) \subset \cV(\alpha)$ holds.
\item[(3)] For any $\alpha \in \kappa \cap \Lim$, $\cV(\alpha) = \langle \bigcup_{\beta \in \alpha} \cV(\beta) \rangle$ holds.
\item[(4)] The equality $V = \bigcup_{\beta \in \kappa} \cV(\beta)$ holds.
\ei
We denote by $\FF(V)$ the set of free filtrations of $V$. 
\end{dfn}

\begin{rmk}
Let $V \in \Ban(k)$. Set $\kappa \coloneqq \rank(V)$. The notion of a free filtration of $V$ is equivalent to that of a functor $\cV \colon \kappa \to \cL_{< \kappa}(V)$ satisfying the following:
\bi
\item[(1)] For any $\alpha \in \kappa$, $\cV(\alpha)$ is free.
\item[(3)'] The functor $\cV$ is cocontinuous as a functor $\kappa \to L$, where $L$ denotes the set of closed $k$-linear subspaces of $V$ containing $\cV(0)$ if $\kappa > 0$ partially ordered by the inclusions.
\item[(4)'] The canonical morphism $\varinjlim (F_{\Comp} \circ \cV) \to F_{\Comp}(V)$ is an isomorphism.
\ei
Whenever we refer to the cocontinuity of $\cV$ in this paper, we intend the cocontinuity in the sense of (3)'.
\end{rmk}

For a well-ordered set $X$, {\it the enumeration of $X$} is the unique order-preserving bijective map from the order type of $X$ to $X$. For any $\kappa \in \Reg$, every unbounded subset $C$ of $\kappa$ is of cardinality $\kappa$, and the enumeration of $C$ well-ordered by the membership relation is a functor $\kappa \to C$ which is a categorical isomorphism.

\begin{prp}
\label{club restriction}
Let $V \in \Ban(k)$ and $\cV \in \FF(V)$. Set $\kappa \coloneqq \rank(V)$. If $\kappa \in \Reg$, for any club set $C$ of $\kappa$, the composite functor $\kappa \to \cL_{< \kappa}(V)$ of the enumeration of $C$ and $\cV$ is a free filtration of $V$.
\end{prp}

\begin{proof}
We denote by $\iota$ the enumeration of $C$. Set $\cV' \coloneqq \cV \circ \iota$. For any $\alpha \in \kappa$, $\cV'(\alpha) = \cV(\iota(\alpha))$ is free. Since $C$ is closed, $\iota$ is Scott continuous as an order-preserving map, and hence is cocontinuous as a functor. By the cocontinuity of $\iota$ and $\cV$, $\cV'$ is cocontinuous. Since $C$ is unbounded, $\iota$ is cofinal, and hence the canonical morphism $\varinjlim (F_{\Comp} \circ \cV') \to \varinjlim (F_{\Comp} \circ \cV)$ is an isomorphism. Therefore the canonical morphism $\varinjlim (F_{\Comp} \circ \cV') \to F_{\Comp}(V)$ is an isomorphism.
\end{proof}

We give a characterisation of almost free Banach $k$-vector spaces in terms of a free filtration, as an analogue of the characterisations of almost free Abelian groups (cf.\ \cite{Fuc15} Lemma 3/8.6 (a)) and almost $\Sigma$-cyclic groups (cf.\ \cite{CO25} Proposition 4.6).

\begin{prp}
\label{almost free}
Let $V \in \Ban(k)$. Set $\kappa \coloneqq \rank(V)$. If $\# V = \kappa \in \Reg$, then the following are equivalent:
\bi
\item[(1)] The Banach $k$-vector space $V$ is almost free.
\item[(2)] The set $\FF(V)$ is non-empty.
\ei
\end{prp}

\begin{proof}
By $\kappa \in \Reg$, $\kappa$ is $\aleph_1$-directed. Assume (1). By the definition of $\kappa$, there exists some $S \in V^{\kappa}$ such that $V = \langle S[\kappa] \rangle$. We define a functor $S' \colon \kappa \to \cP_{< \kappa}(V)$ by
\be
S'(\alpha) \coloneqq S[\alpha].
\ee
For any $\alpha \in \kappa \cap \Lim$, we have
\be
S'(\alpha) = S[\alpha] = S \left[ \bigcup_{\beta \in \alpha} \beta \right] = \bigcup_{\beta \in \alpha} S[\beta] = \bigcup_{\beta \in \alpha} S'(\beta).
\ee
This implies the cocontinuity of $S'$. We define a functor $\cV \colon \kappa \to \cL_{< \kappa}(V)$ by
\be
\cV(\alpha) \coloneqq \langle S'(\alpha) \rangle.
\ee
Since $V$ is $\kappa$-free, $\cV(\alpha)$ is free for any $\alpha \in \kappa$. Since $S'$ is cocontinuous and $\Comp$ is left adjoint to $F_{\Comp}$, $\cV$ is cocontinuous. We denote by $\iota$ the canonical morphism $\varinjlim (F_{\Comp} \circ \cV) \to F_{\Comp}(V)$. Since $F_{\Comp} \circ \cV$ is a system of inclusions of $k$-linear subspaces of $V$, $\iota$ is an isometry, and the image of $\iota$ is closed by the $\aleph_1$-directedness of $\kappa$ and Proposition \ref{completeness of aleph_1-directed limit}. By $V = \langle S[\kappa] \rangle$, the image of $\iota$ is dense in $V$. Therefore, $\iota$ is an isomorphism, and we obtain $\cV \in \FF(V)$.

\vs
Assume (2). Take $\cV \in \FF(V)$. Let $V' \in \cL_{< \kappa}(V)$. Take $S \in \cP_{< \kappa}(V)$ with $V' = \langle S \rangle$. For each $v \in S$, we denote by $\alpha_v \in \kappa$ the minimum of $\alpha \in \kappa$ with $v \in \cV(\alpha)$. Set $\alpha \coloneqq \sup_{v \in S} \alpha_v$. By $\# S < \kappa$ and the regularity of $\kappa$, we have $\alpha < \kappa$. This implies $S \subset \cV(\alpha)$, and hence $V' \subset \cV(\alpha)$. Since $\cV(\alpha)$ is free, so is $V'$ by \cite{Sch99} Theorem 5 (ii). Therefore, $V$ is $\kappa$-free, i.e.\ $V$ is almost free.
\end{proof}

%% file: aleph_1_Strongly_Compact_Case.tex
\section{$\aleph_1$-strongly Compact Case}
\label{aleph_1-strongly Compact Case}

We denote by $\lambda_k$ the successor cardinal of $\max \ens{\#(k/O_k),\aleph_0}$. If the valuation of $k$ is non-trivial, then $\lambda_k$ coincides with $\lambda$ in Corollary \ref{accessibility}. For example, if $k$ is separable, then we have $\lambda_k = \aleph_1$. We give a non-Archimedean Banach space analogue of the implication from almost freeness to freeness when the cardinality is regular and $\aleph_1$-strongly compact (cf.\ \cite{CO25} \S 1 (1)).

\begin{thm}[Theorem \ref{strongly compact intro}]
\label{strongly compact}
Let $V \in \Ban(k)$ with $\# V = \rank(V)$. If $\rank(V)$ is regular and $\lambda_k$-strongly compact, then the following are equivalent:
\bi
\item[(1)] The Banach $k$-vector space $V$ is free.
\item[(2)] The Banach $k$-vector space $V$ is almost free.
\ei
\end{thm}

We prove Theorem \ref{strongly compact} in a way parallel to the $\Sigma$-cyclic counterpart \cite{CO25} Theorem 1.1. For this purpose, we prepare a non-Archimedean Banach space analogue of \cite{CO25} Lemma 3.2.

\begin{lmm}
\label{ultrapower of basis}
Let $I \in \Set$ with a $\lambda_k$-complete ultrafilter $\cU$ on $I$, and $\cV \in \Ban(k)^I$. If $\cV(i)$ is free for all $i \in I$, then $(\prod_{I \in I} \cV(i))/\cU$ is free.
\end{lmm}

\begin{proof}
Since the bounded direct product and the quotient by a closed subspace preserve the triviality of a norm, it suffices to consider the case where the valuation of $k$ is non-trivial by Proposition \ref{trivial}.

\vs
For each $i \in I$, take an orthonormal Schauder basis $E_i$ of $\cV(i)$. Set
\be
E & \coloneqq & \set{v \in \prod_{i \in I} \cV(i)}{\set{i \in I}{v(i) \in E_i} \in \cU} \\
E_{\cU} & \coloneqq & \set{[e]_{\cU}}{e \in E}.
\ee
It suffices to show that $E_{\cU}$ is an orthonormal Schauder basis of $(\prod_{i \in I} \cV(i))/\cU$. Let $f \in E_{\cU}$. Take $e_f \in E$ with $[e_f]_{\cU} = f$, and set $I'_f \coloneqq \set{i \in I}{e_f(i) \in E_i}$. We have $I'_f \in \cU$ by $e_f \in E$, and $I'_f \subset \set{i \in I}{\n{e_f(i)} = 1}$ because $E_i$ is an orthonormal Schauder basis of $\cV(i)$. We obtain $\n{f} = \n{[e_f]_{\cU}} = 1$ by $I'_f \in \cU$ and Corollary \ref{equivalence relation} (2). 

\vs
Let $c \in \rC_0(E_{\cU},k)$. We show $\n{\sum_{f \in E_{\cU}} c(f) f} = \n{c}$. Set $S \coloneqq \set{f \in E_{\cU}}{c(f) \neq 0}$. By $c \in \rC_0(E_{\cU},k)$, we have $\# S \leq \aleph_0$, and hence $\# \cP_{= 2}(S) \leq \aleph_0$. Set $I'_{f,f'} \coloneqq \set{i \in I}{e_f(i) \neq e_{f'}(i)}$ for each $(f,f') \in S^2$. For any $(f,f') \in S^2$ with $f \neq f'$, we have $I'_{f,f'} \in \cU$ by the definition of the quotient map $\prod_{i \in I} \cV(i) \twoheadrightarrow (\prod_{i \in I} \cV(i))/\cU$. Set
\be
I' \coloneqq \bigcap \left(\ens{I} \cup \set{I'_f}{f \in S} \cup \set{I'_{f,f'}}{(f,f') \in S^2 \land f \neq f'} \right).
\ee
We have $I' \in \cU$ by $\# \cP_{= 2}(S) \leq \aleph_0$ and the $\aleph_1$-completeness of $\cU$. For any $i \in I'$, we have $e_f(i) \in E_i$ for any $f \in S$ and $e_f(i) \neq e_{f'}(i)$ for any $(f,f') \in S^2$ with $f \neq f'$ by the definition of $I'$, and hence
\be
\n{\sum_{f \in E_{\cU}} c(f) e_f(i)} = \n{\sum_{f \in S} c(f) e_f(i)} = \n{c \upharpoonright S} = \n{c}
\ee
because $E_i$ is an orthonormal Schauder basis of $\cV(i)$. This implies
\be
\set{i \in I}{\n{\sum_{f \in S} c(f) e_f(i)} = \n{c}} \in \cU
\ee
by $I' \in \cU$, and hence $\n{\sum_{f \in E_{\cU}} c(f) f} = \n{c}$ by Corollary \ref{equivalence relation} (2).

\vs
Let $\ol{v} \in (\prod_{i \in I} \cV(i))/\cU$. We show that there exists some $c \in \rC_0(E_{\cU},k)$ such that $\sum_{f \in E_{\cU}} c(f) f = \ol{v}$. By the completeness of $\rC_0(E_{\cU},k)$, it suffices to show that for all $\epsilon \in \R_{> 0}$, there exists some $c \in k^{\oplus E_{\cU}}$ such that $\n{\ol{v} - \sum_{f \in E_{\cU}} c(f) f} < \epsilon$.

\vs
Set $B \coloneqq \set{c \in k}{\v{c} < \epsilon}$. Take a complete representative $C$ of the canonical projection $k \twoheadrightarrow k/B$. Since the valuation of $k$ is non-trivial, we have $\# C = \#(k/B) = \#(k/O_k) < \lambda_k$, and hence $\#(\bigsqcup_{n \in \omega} C^n) < \lambda_k$ by $\lambda_k > \aleph_0$.

\vs
Take a representative $v \in \prod_{i \in I} \cV(i)$ of $\ol{v}$. For each $n \in \omega$ and $c \in C^n$, we denote by $I_{n,c}$ the set of $i \in I$ for which there exists some injective map $e \colon \omega_{< n} \hookrightarrow E_i$ such that $\n{v(i) - \sum_{h \in \omega_{< n}} c(h) e(h)} < \epsilon$. We show $I = \bigcup_{n \in \omega} \bigcup_{c \in C^n} I_{n,c}$.

\vs
Let $i \in I$. Since $E_i$ is an orthonormal Schauder basis of $\cV(i)$, there exists some $c \in \rC_0(E_i,k)$ such that $v(i) = \sum_{e \in E_i} c(e) e$. Set $E'_i \coloneqq \set{e \in E_i}{\v{c(e)} \geq \epsilon}$ and $n \coloneqq \# E'_i$. By the definition of $\rC_0(E_i,k)$, we have $n < \aleph_0$. Take an injective map $e \colon \omega_{< n} \hookrightarrow E_i$ onto $E'_i$. We have
\be
\n{v(i) - \sum_{h \in \omega_{< n}} c(e(h)) e(h)} = \n{\sum_{e \in E_i} c(e) e - \sum_{e \in E'_i} c(e) e} = \n{\sum_{e \in E_i \setminus E'_i} c(e) e} < \epsilon,
\ee
and hence $i \in I_{n,c \circ e}$.

\vs
By $\bigcup_{n \in \omega} \bigcup_{c \in C^n} I_{n,c} = I \in \cU$, $\#(\bigsqcup_{n \in \omega} C^n) < \lambda_k$, and the $\lambda_k$-completeness of $\cU$, there exists a pair $(n,c)$ of $n \in \omega$ and $c \in C^n$ such that $I_{n,c} \in \cU$. For each $i \in I_{n,c}$, take an injective map $e_i \colon \omega_{< n} \hookrightarrow E_i$ such that $\n{v(i) - \sum_{h \in \omega_{< n}} c(h) e_i(h)} < \epsilon$. For each $h \in \omega_{< n}$, we define $e'_h \in E$ by
\be
e'_h(i) \coloneqq 
\left\{
\begin{array}{ll}
e_i(h) & (i \in I_{n,c}) \\
0 & (i \notin I_{n,c})
\end{array}
\right.
\ee
For any $i \in I_{n,c}$, we have
\be
\n{v(i) - \sum_{h \in \omega_{< n}} c(h) e'_h(i)} = \n{v(i) - \sum_{h \in \omega_{< n}} c(h) e_i(h)} < \epsilon.
\ee
This implies $\n{\ol{v} - \sum_{h \in \omega_{< n}} c(h) [e'_h]_{\cU}} < \epsilon$.
\end{proof}

\begin{proof}[Proof of Theorem \ref{strongly compact}]
Set $\kappa \coloneqq \rank(V)$. By Proposition \ref{free implies almost free}, (1) implies (2). Assume (2). For $V' \in \cL_{< \kappa}(V)$, set $\cU_{V'} \coloneqq \set{V'' \in \cL_{< \kappa}(V)}{V' \subset V''}$. Set
\be
\cF \coloneqq \set{U \in \cP(\cL_{< \kappa}(V))}{\exists V' \in \cL_{< \kappa}(V)[\cU_{V'} \subset U]}.
\ee
For any $V' \in \cL_{< \kappa}(V)$, we have $V' \in \cU_{V'}$ and hence $\cU_{V'} \neq \emptyset$. For any $S \in \cP_{< \kappa}(\cL_{< \kappa}(V)) \setminus \ens{\emptyset}$, we have $\langle \bigcup S \rangle \in \cL_{< \kappa}(V)$ by the definition of $\cL_{< \kappa}(V)$ and the regularity of $\kappa$, and $\cU_{\langle \bigcup S \rangle} \subset \bigcap_{V' \in S} \cU_{V'}$. Therefore, $\cF$ is a $\kappa$-complete filter on $\cL_{< \kappa}(V)$.

\vs
Since $\kappa$ is $\lambda_k$-strongly compact and $\cF$ is a $\kappa$-complete filter on $\cL_{< \kappa}(V)$, there exists some $\lambda_k$-complete ultrafilter $\cU$ on $\cL_{< \kappa}(V)$ containing $\cF$. For each $v \in V$, we define $\iota_v \in \prod_{V' \in \cL_{< \kappa}(V)} V'$ by
\be
\iota_v(V') \coloneqq
\left\{
\begin{array}{ll}
v & (v \in V') \\
0 & (v \notin V')
\end{array}
\right..
\ee
We define a map $\iota \colon V \to (\prod_{V' \in \cL_{< \kappa}(V)} V')/\cU$ by
\be
\iota(v) \coloneqq [(\iota_v(V'))_{V' \in \cL_{< \kappa}(V)}]_{\cU}.
\ee
We show that $\iota$ is an isometric $k$-linear homomorphism.

\vs
Let $(v_0,v_1) \in V^2$. We show $\iota(v_0 + v_1) = \iota(v_0) + \iota(v_1)$. Set $V' \coloneqq \langle \ens{v_0,v_1} \rangle \in \cL_{< \kappa}(V)$. For any $V'' \in \cU_{V'}$, we have
\be
\iota_{v_0 + v_1}(V'') = v_0 + v_1 = \iota_{v_0}(V'') + \iota_{v_1}(V'').
\ee
Therefore, we obtain $\iota(v_0 + v_1) = \iota(v_0) + \iota(v_1)$ by $\cU_{V'} \in \cF \subset \cU$.

\vs
Let $(c,v) \in k \times V$. We show $\iota(cv) = c \iota(v)$. Set $V' \coloneqq \langle \ens{v} \rangle \in \cL_{< \kappa}(V)$. For any $V'' \in \cU_{V'}$, we have
\be
\iota_{cv}(V'') = cv = c \iota_{v}(V'').
\ee
Therefore, we obtain $\iota(cv) = c \iota(v)$ by $\cU_{V'} \in \cF \subset \cU$.

\vs
Let $v \in V$. We show $\n{\iota(v)} = \n{v}$. Set $V' \coloneqq \langle \ens{v} \rangle \in \cL_{< \kappa}(V)$. For any $V'' \in \cU_{V'}$, we have
\be
\n{\iota_v(V'')} = \n{v}.
\ee
Therefore, we obtain $\n{\iota(v)} = \n{v}$ by $\cU_{V'} \in \cF \subset \cU$ and Corollary \ref{equivalence relation} (2).

\vs
Since $V$ is complete and $\iota$ is an isometric $k$-linear homomorphism, $\iota[V]$ is a closed $k$-linear subspace of $(\prod_{V' \in \cL_{< \kappa}(V)} V')/\cU$ isomorphic to $V$. By Lemma \ref{ultrapower of basis} and the almost freeness of $V$, $(\prod_{V' \in \cL_{< \kappa}(V)} V')/\cU$ is free. Therefore, $\iota[V]$ is free by \cite{Sch99} Theorem 5 (ii). We conclude that $V$ is free.
\end{proof}

\begin{crl}
\label{aleph_1 strongly compact}
Let $V \in \Ban(k)$ with $\# V = \rank(V)$. If $k$ is a closed subfield of the completion $\Cp$ of an algebraic closure of the $p$-adic field $\Qp$ for a prime number $p$ and $\rank(V)$ is regular and $\aleph_1$-strongly compact, then the following are equivalent:
\bi
\item[(1)] The Banach $k$-vector space $V$ is free.
\item[(2)] The Banach $k$-vector space $V$ is almost free.
\ei
\end{crl}

\begin{proof}
The assertion immediately follows from Theorem \ref{strongly compact} by $\lambda_k = \aleph_1$.
\end{proof}

%% file: Weakly_Compact_Case.tex
\section{Weakly Compact Case}
\label{Weakly Compact Case}

We give a non-Archimedean Banach space analogue of the implication from almost freeness to freeness when the cardinality is weakly compact (cf.\ \cite{CO25} \S 1 (2)).

\begin{thm}[Theorem \ref{weakly compact intro}]
\label{weakly compact}
Let $V \in \Ban(k)$. If $\# V = \rank(V)$ and $\rank(V)$ is weakly compact, then the following are equivalent:
\bi
\item[(1)] The Banach $k$-vector space $V$ is free.
\item[(2)] The Banach $k$-vector space $V$ is almost free.
\ei
\end{thm}

We prove Theorem \ref{weakly compact} in a way parallel to the $\Sigma$-cyclic counterpart \cite{CO25} Theorem 1.2. For this purpose, we prepare several lemmata on free filtrations. First, we give a non-Archimedean Banach space analogue of \cite{Fuc15} Lemma 1/4.3.

\begin{lmm}
\label{filter equivalence}
Let $V \in \Ban(k)$. Set $\kappa \coloneqq \rank(V)$. If $\kappa \in \Reg$, then for all $(\cV_0,\cV_1) \in \FF(V)^2$, $\set{\alpha \in \kappa}{\cV_0(\alpha) = \cV_1(\alpha)}$ is a club set of $\kappa$.
\end{lmm}

\begin{proof}
The closedness follows from the cocontinuity of $\cV_0$ and $\cV_1$. The unboundedness immediately follows from back-and-forth argument (cf.\ the proof of \cite{Fuc15} Lemma 1/4.3). Indeed, for all $\alpha \in \kappa$, there exists a functor $\iota \colon \aleph_0 \to \kappa$ such that $\alpha = \iota(0)$ and $\cV_0(\iota(2i)) \subsetneq \cV_1(\iota(2i+1)) \subsetneq \cV_0(\iota(2i+2))$ for any $i \in \aleph_0$ by $(\cV_0,\cV_1) \in \FF(V)^2$ and the regularity of $\kappa$, and such an $\iota$ converges to some $\alpha' \in \kappa$ with $\cV_0(\alpha') = \cV_1(\alpha')$ again by $(\cV_0,\cV_1) \in \FF(V)^2$ and the regularity of $\kappa$.
\end{proof}

In Abelian group theory, the following lifting property plays an important role in the study of a filtration: For any subgroup $B$ of an Abelian group $A$, $B$ is a direct summand of $A$ if $A/B$ is free (cf.\ \cite{Fuc15} Theorem 3/1.5) or if $B$ is a pure subgroup of $A$ and $A/B$ is $\Sigma$-cyclic (cf.\ \cite{CO25} Theorem 4.9). Although a similar lifting property holds for a free Banach $k$-vector space if we consider bounded $k$-linear homomorphisms (cf.\ \cite{BM23} p.\ 430), the same does not hold for contracting $k$-linear homomorphisms. Therefore, we formulate an alternative notion instead of a lifting property.

\begin{dfn}
Let $V \in \Ban(k)$ with a closed $k$-linear subspace $W \subset V$. Set $\kappa \coloneqq \rank(V)$. A {\it supplement of $W$ in $V$} is a closed $k$-linear subspace $W^{\perp}$ such that the canonical morphism $W \oplus W^{\perp} \to V$ is an isomorphism (in $\Ban(k)$ by the convention in \S \ref{Convention}). We say that $W$ is a {\it direct summand of $V$} (resp.\ {\it cofree direct summand of $V$}) if it admits a supplement in $V$ (resp.\ free supplement in $V$), and is an {\it almost cofree direct summand of $V$} if for any $W' \in \cL_{< \kappa}(V)$ with $W \subset W'$, $W$ is a cofree direct summand of $W'$. We denote by $\CF(V)$ (resp.\ $\ACF(V)$) the set of cofree (resp.\ almost cofree) direct summands of $V$.
\end{dfn}

\begin{lmm}
\label{supplement isomophism}
Let $V \in \Ban(k)$ with a direct summand $W$. For any supplement $W^{\perp}$ of $W$ in $V$, the restriction $W^{\perp} \to V/W$ of the canonical projection $V \twoheadrightarrow V/W$ is an isomorphism.
\end{lmm}

\begin{proof}
We denote by $\iota \colon W^{\perp} \to V/W$ the given morphism. The bijectivity of $\iota$ follows from the well-known fact on a direct summand in the algebraic setting. It suffices to show that $\iota$ is an isometry. Let $w' \in W^{\perp}$. For any $w \in W$, we have $\n{w + w'} = \max \ens{\n{w},\n{w'}} \geq \n{w'}$ by the definition of a supplement. This implies $\n{w' + W} = \n{w'}$.
\end{proof}

\begin{lmm}
\label{cofree direct summand preserving}
Let $V \in \Ban(k)$ with a cofree direct summand $W$. For any closed $k$-linear subspace $W'$ of $V$ with $W \subset W'$, $W$ is a cofree direct summand of $W'$. 
\end{lmm}

\begin{proof}
Take a free supplement $W^{\perp}$ of $W$ in $V$. We denote by $\iota \colon W^{\perp} \to V/W$ the isomorphism in Lemma \ref{supplement isomophism}. Set $W^{\perp}_0 \coloneqq \iota^{-1}[W'/W]$. We denote by $f$ the canonical morphism $W \oplus W^{\perp}_0 \to W'$. The bijectivity of $f$ follows from the well-known fact on a direct summand in the algebraic setting. Since $f$ is the restriction of the canonical isomorphism $W \oplus W^{\perp} \to V$, it is an isometry. Since $W^{\perp}_0$ is a closed $k$-linear subspace of the free Banach $k$-vector space $W^{\perp}$, it is free (cf.\ \cite{Sch99} Theorem 5 (ii)).
\end{proof}

\begin{lmm}
\label{almost cofree vs cofree}
Let $V \in \Ban(k)$ and $\cV \in \FF(V)$. Set $\kappa \coloneqq \rank(V)$. If $\kappa$ is regular, then for any $\alpha \in \kappa$, the following are equivalent:
\bi
\item[(1)] The relation $\cV(\alpha) \in \ACF(V)$ holds.
\item[(2)] For any $\alpha' \in \kappa \setminus \alpha$, $\cV(\alpha) \in \CF(\cV(\alpha'))$ holds.
\item[(3)] The set $\set{\alpha' \in \kappa \setminus \alpha}{\cV(\alpha) \in \CF(\cV(\alpha'))}$ is unbounded in $\kappa$.
\ei
\end{lmm}

\begin{proof}
The implication from (1) to (2) follows from $\cV(\alpha) \subset \cV(\alpha')$ and $\cV(\alpha') \in \cL_{< \kappa}(V)$. The implication from (2) to (3) follows from the unboundedness of $\kappa \setminus \alpha$ in $\kappa$. Assume (3). Set $C \coloneqq \set{\alpha' \in \kappa \setminus \alpha}{\cV(\alpha) \in \CF(\cV(\alpha'))}$. Let $W \in \cL_{< \kappa}(V)$. Take $S \in \cP_{< \kappa}(W)$ with $W = \langle S \rangle$. By $\# S < \kappa$, the regularity of $\kappa$, and the unboundedness of $C$ in $\kappa$, there exists some $\alpha' \in C$ such that $S \subset \cV(\alpha')$ and hence $W \subset \cV(\alpha')$. By $\alpha' \in C$, we have $\cV(\alpha) \in \CF(\cV(\alpha'))$. Therefore, we obtain $\cV(\alpha) \in \CF(W)$ by Lemma \ref{cofree direct summand preserving} applied to $(\cV(\alpha'),\cV(\alpha))$.
\end{proof}

\begin{lmm}
\label{cofree characterisation}
Let $V \in \Ban(k)$. Set $\kappa \coloneqq \rank(V)$. If $\kappa \in \Reg$ and $\FF(V) \cap \ACF(V)^{\kappa} \neq \emptyset$, then $V$ is free.
\end{lmm}

\begin{proof}
Take $\cV \in \FF(V) \cap \ACF(V)^{\kappa}$. Let $\alpha \in \kappa$. If $\alpha \in \Suc$, then we have $\cV(\alpha^{-}) \in \CF(\cV(\alpha))$ by $\cV(\alpha^{-}) \in \ACF(V)$ and Lemma \ref{almost cofree vs cofree}, and fix a free supplement $\cV(\alpha^{-})^{\perp}$ of $\cV(\alpha^{-})$ in $\cV(\alpha)$. Set
\be
W_{\alpha} \coloneqq 
\left\{
\begin{array}{ll}
\cV(0) & (\alpha = 0) \\
\cV(\alpha^{-})^{\perp} & (\alpha \in \Suc) \\
\ens{0} & (\alpha \in \Lim)
\end{array}
\right..
\ee
Then $W_{\alpha}$ is a free closed $k$-linear subspace of $\cV(\alpha)$ of rank $< \kappa$ by definition.

\vs
It suffices to show that the canonical morphism
\be
\iota \colon \bigoplus_{\beta \in \alpha} W_{\alpha} \to V
\ee
is an isomorphism, because the domain is the completed direct sum of free Banach $k$-vector spaces. By the definition of $(W_{\alpha})_{\alpha \in \kappa}$ using free summands, $\iota$ is an isometry. Therefore, it suffices to show that the image of $\iota$ is dense in $V$. By $\bigcup_{\alpha \in \kappa} \cV(\alpha) = V$, it suffices to show that $\bigcup_{\beta \in \alpha + 1} W_{\beta}$ generates a dense $k$-vector subspace of $\cV(\alpha)$ for any $\alpha \in \kappa$ by induction on $\alpha$.

\vs
If $\alpha = 0$, then we have
\be
\left\langle \bigcup_{\beta \in \alpha + 1} W_{\beta} \right\rangle = \left\langle W_0 \right\rangle = \left\langle \cV(0) \right\rangle = \cV(0)
\ee
by the definition of $W_{\alpha}$. If $\alpha \in \Suc$, then we have
\be
\left\langle \bigcup_{\beta \in \alpha + 1} W_{\beta} \right\rangle = \left\langle \left\langle \bigcup_{\beta \in \alpha} W_{\beta} \right\rangle \cup W_{\alpha} \right\rangle = \left\langle \cV(\alpha^{-}) \cup \cV(\alpha^{-})^{\perp} \right\rangle = \cV(\alpha).\
\ee
by the induction hypothesis and the definition of $W_{\alpha}$. If $\alpha \in \Lim$, then we have
\be
\left\langle \bigcup_{\beta \in \alpha + 1} W_{\beta} \right\rangle = \left\langle \bigcup_{\alpha' \in \alpha} \left\langle \bigcup_{\beta \in \alpha' + 1} W_{\beta} \right\rangle \cup W_{\alpha} \right\rangle = \left\langle \bigcup_{\beta \in \alpha} \cV(\beta) \cup \ens{0} \right\rangle = \cV(\alpha).\
\ee
by the induction hypothesis, the definition of $W_{\alpha}$, and $\cV \in \FF(V)$.
\end{proof}

\begin{dfn}
Let $V \in \Ban(k)$ and $\cV \in \FF(V)$. Set $\kappa \coloneqq \rank(V)$. We set
\be
E_{\cV} \coloneqq \set{\alpha \in \kappa}{\cV(\alpha) \notin \ACF(V)}.
\ee
For $\alpha \in \kappa$, we set
\be
S_{\cV,\alpha} \coloneqq \set{\alpha' \in \kappa \setminus \alpha}{\cV(\alpha) \notin \CF(\cV(\alpha'))}.
\ee
\end{dfn}

We give a non-Archimedean Banach space analogue of \cite{CO25} Remark 4.9.

\begin{lmm}
\label{exceptional set characterisation}
Let $V \in \Ban(k)$ and $\cV \in \FF(V)$. Set $\kappa \coloneqq \rank(V)$. If $\kappa$ is regular, then for any $\alpha \in \kappa$, the following are equivalent:
\bi
\item[(1)] The relation $\alpha \in E_{\cV}$ holds.
\item[(2)] The set $S_{\cV,\alpha}$ is stationary in $\kappa$.
\ei
\end{lmm}

\begin{proof}
It suffices to show that $\cV(\alpha) \in \ACF(V)$ holds if and only if $S_{\cV,\alpha}$ is not stationary in $\kappa$. For this purpose, it suffices to show that $\cV(\alpha) \in \ACF(V)$ holds if $S_{\cV,\alpha}$ is not stationary in $\kappa$ by Lemma \ref{almost cofree vs cofree}.

\vs
Since $S_{\cV,\alpha}$ is not stationary in $\kappa$, there exists a club set $C$ of $\kappa$ such that $S_{\cV,\alpha} \cap C = \emptyset$. This implies that for any $\alpha' \in C \setminus \alpha$, $\cV(\alpha) \in \CF(\cV(\alpha'))$ holds. Since $C$ is unbounded, we have $\cV(\alpha) \in \ACF(V)$ by Lemma \ref{almost cofree vs cofree}.
\end{proof}

We give a lemma which plays a role similar to Eklof--Shelah criterion (cf.\ \cite{Fuc15} Theorem 3/7.5 for free Abelian groups and \cite{CO25} Theorem 4.11 for $\Sigma$-cyclic Abelian groups), which is originally based on lifting properties of quotient homomorphisms.

\begin{lmm}
\label{exceptional set criterion}
Let $V \in \Ban(k)$. Set $\kappa \coloneqq \rank(V)$. If $\# V = \kappa \in \Reg$, then the following are equivalent:
\bi
\item[(1)] The Banach $k$-vector space $V$ is free.
\item[(2)] There exists some $\cV \in \FF(V)$ such that $E_{\cV}$ is not stationary in $\kappa$.
\item[(3)] The Banach $k$-vector space $V$ is almost free, and for any $\cV \in \FF(V)$, $E_{\cV}$ is not stationary in $\kappa$.
\ei
\end{lmm}

\begin{proof}
Assume (1). Take an orthonormal Schauder basis $E \subset V$ of $V$. By Proposition \ref{uniqueness of dimension}, we have $\# E = \kappa$. Take a bijective map $e \colon \kappa \to E$. We define a functor $\cV \colon \kappa \to \Ban(k)$ by
\be
\cV(\alpha) \coloneqq \langle e[\alpha] \rangle.
\ee
For any $\alpha \in \kappa$, $e[\alpha] \subset E$ is an orthonormal Schauder basis of $\cV(\alpha)$, and hence $\cV(\alpha)$ is free. By the definition, $\cV$ is cocontinuous. By Proposition \ref{completeness of aleph_1-directed limit} (2) and Proposition \ref{L is directed}, the canonical morphism $\varinjlim_{\alpha \in \kappa} (F_{\Comp} \circ \cV(\alpha)) \to F_{\Comp}(V)$ is an isomorphism. Therefore, we obtain $\cV \in \FF(V)$. For any $\alpha \in \kappa$, the map
\be
\rC_0(\kappa \setminus \alpha,k) & \to & V \\
f & \mapsto & \sum_{\beta \in \kappa \setminus \alpha} f(\beta) e(\beta)
\ee
is an isometry onto a free supplement of $\cV(\alpha)$ in $V$, and hence $\alpha \in \ACF(V)$ holds by Lemma \ref{cofree direct summand preserving}. Therefore, we have $E_{\cV} = \emptyset$. This implies that $E_{\cV}$ is not stationary in $\kappa$.

\vs
Assume (2). Take $\cV_0 \in \FF(V)$ such that $E_{\cV_0}$ is not stationary in $\kappa$. By Proposition \ref{almost free}, $V$ is almost free. Let $\cV_1 \in \FF(V)$. Set $C \coloneqq \set{\alpha \in \kappa}{\cV_0(\alpha) = \cV_1(\alpha)}$.

\vs
By Lemma \ref{filter equivalence}, $C$ is a club set of $\kappa$. Since $E_{\cV_0}$ is not stationary in $\kappa$, there exists a club set $C'$ of $\kappa$ such that $C' \cap E_{\cV_0} = \emptyset$. We have
\be
(C \cap C') \cap E_{\cV_1} = C' \cap (C \cap E_{\cV_1}) = C' \cap (C \cap E_{\cV_0}) \subset C' \cap E_{\cV_0} = \emptyset,
\ee
and hence $E_{\cV_1}$ is not stationary in $\kappa$ because $C \cap C'$ is a club set of $\kappa$ (cf.\ \cite{Fuc15} Lemma 1/4.1).

\vs
Assume (3). Take $\cV \in \FF(V)$, which exists by Proposition \ref{almost free}. Since $E_{\cV}$ is not stationary in $\kappa$, there exists a club set $C$ of $\kappa$ such that $E_{\cV} \cap C = \emptyset$. We denote by $\cV'$ the composite functor $\kappa \to \cL_{< \kappa}(V)$ of the enumeration $\kappa \to C$ of $C$ and $\cV$. By Proposition \ref{club restriction}, we have $\cV' \in \FF(V)$. By $E_{\cV} \cap C = \emptyset$, we have $\cV' \in \ACF(V)^{\kappa}$. Therefore, $V$ is free by Lemma \ref{cofree characterisation}.
\end{proof}

\begin{proof}[Proof of Theorem \ref{weakly compact}]
Set $\kappa \coloneqq \rank(V)$. By Proposition \ref{free implies almost free}, (1) implies (2). Assume (2). Take $\cV \in \FF(V)$, which exists by Proposition \ref{almost free}. By $\kappa = 1 + \kappa$, we may assume $\cV(0) = \ens{0}$.

\vs
Assume that $V$ is not free. Then $E_{\cV}$ is stationary in $\kappa$ by Lemma \ref{exceptional set criterion}. For each $\alpha \in \kappa$, set
\be
S'_{\alpha} \coloneqq
\left\{
\begin{array}{ll}
S_{\cV,\alpha} & (\alpha \in E_{\cV}) \\
E_{\cV} & (\alpha \notin E_{\cV})
\end{array}
\right..
\ee
By Lemma \ref{exceptional set characterisation}, $S'_{\alpha}$ is stationary in $\kappa$ for any $\alpha \in \kappa$. By the stationary reflection property of $\kappa$ (cf.\ \cite{CO25} Lemma 4.1), there exists some $T \in \cP(\kappa)$ satisfying the following:
\bi
\item[(3)] The set $T$ is stationary in $\kappa$.
\item[(4)] For any $\lambda \in T$, $\lambda \in \Reg$ holds.
\item[(5)] For any $\lambda \in T$ and $\alpha \in \lambda$, $S'_{\alpha} \cap \lambda$ is stationary in $\lambda$.
\ei
Take $\lambda \in T$, which exists by (3). For any $\alpha \in \kappa$, we have $\cV(0) = \ens{0} \in \CF(\cV(\alpha))$, and hence $\alpha \notin S_{\cV,0}$. This implies $S_{\cV,0} = \emptyset$, and hence $0 \notin E_{\cV}$ by Lemma \ref{exceptional set characterisation}. Therefore, we have $S'_0 = E_{\cV}$, and hence $E_{\cV} \cap \lambda$ is stationary in $\lambda$ by (5) applied to $\alpha = 0$. By (4), Proposition \ref{completeness of aleph_1-directed limit} (2), and Proposition \ref{L is directed}, we have $\cV(\lambda) = \bigcup_{\alpha \in \lambda} \cV(\alpha)$, and hence $\cV \upharpoonright \lambda \in \FF(\cV(\lambda))$. Let $\alpha \in E_{\cV} \cap \lambda$. We have
\be
S_{\cV \upharpoonright \lambda,\alpha} = S_{\cV,\alpha} \cap \lambda = S'_{\alpha} \cap \lambda
\ee
by $\alpha \in E_{\cV}$, and hence $S_{\cV \upharpoonright \lambda}$ is stationary in $\lambda$ by (5). Therefore, we obtain $\alpha \in E_{\cV \upharpoonright \lambda}$ by Lemma \ref{exceptional set characterisation}. This implies $E_{\cV} \cap \lambda \subset E_{\cV \upharpoonright \lambda}$. Since $E_{\cV} \cap \lambda$ is stationary in $\lambda$, so is $E_{\cV \upharpoonright \lambda}$. This contradicts that $\cV(\lambda)$ is free by (4), $\cV \upharpoonright \lambda \in \FF(\cV(\lambda))$, and Lemma \ref{exceptional set criterion}. We conclude that $V$ is free.
\end{proof}

%% file: References.tex
% \newpage
\vspace{0.3in}
\addcontentsline{toc}{section}{Acknowledgements}
\noindent {\Large \bf Acknowledgements}
\vspace{0.2in}

\noindent
We thank an anonymous referee for outstandingly thoughtful suggestions to improve an old version:
\bi
\item[(1)] The referee pointed out several logical mistakes. In particular, the original proofs of Theorem \ref{almost free} and Theorem \ref{strongly compact} had a common critical error in computation of cardinality, which is resolved by the referee's alternative computation. We did not notice the mistakes before the referee's careful reading.
\item[(2)] The referee shared ideas on whether the proof was redundant or should be more detailed, for many propositions. The opinions actually helped us to brush up the original proofs. In particular, the proof of Lemma \ref{cofree characterisation} is drastically shortened by an alternative approach proposed by the referee.
\item[(3)] The referee corrected really many English expressions which were not natural for native speakers or were difficult to parse. In particular, we used ``any'' instead of ``all'' several times in a way misleading to some extent.
\ei
Since we reflected all of them except for a single suggestion to insert ``for all $i \in I'$'' to a location which we could not find, the contribution of the referee appears almost everywhere in the current version.

\vs
We thank K.\ Eda for introducing to us the preceding studies of Specker phenomenon and Reid class, because these led to our study of almost free groups. We thank all people who helped us to learn mathematics and programming. We also thank our family.